\documentclass[]{article} \usepackage{a4wide,epsfig,
  xspace,todonotes,color,amsmath,amssymb}
\newcommand{\bm}[1]{\mbox{$\mathbf #1$}} 
 
\newcommand{\bY}{{\bm Y}} 
\newcommand{\bM}{{\bm M}} 
\newcommand{\bs}{{\bm s}} 
 
\newcommand{\half}{\frac 1 2}
 
\newcommand{\pc}{\mbox{\rm PC}\xspace}
\newcommand{\tpc}{\mbox{\rm$\widetilde{\mbox{PC}}$}\xspace}
\newcommand{\hpc}{\mbox{\rm$\widehat{\mbox{PC}}$}\xspace}
\newcommand{\lb}{\mbox{\rm LB}\xspace}
\newcommand{\slb}{\mbox{{\it{s}}\rm{LB}}\xspace}
\newcommand{\ulb}{\mbox{{\it{u}}\rm{LB}}\xspace}
\newcommand{\ub}{\mbox{\rm UB}\xspace}
\newcommand{\sub}{\mbox{{\it{s}}\rm{UB}}\xspace}
\newcommand{\uub}{\mbox{{\it{u}}\rm{UB}}\xspace}
\newcommand{\nlb}{\mbox{{\it{m}}\rm{LB}}\xspace}
\newcommand{\nub}{\mbox{{\it{m}}\rm{UB}}\xspace}
\newcommand{\oub}{\mbox{{\it{o}}\rm{UB}}\xspace}
\newcommand{\olb}{\mbox{{\it{o}}\rm{LB}}\xspace}
\newcommand{\tM}{\widetilde M\xspace} \newcommand{\tm}{\widetilde
  m\xspace}

\renewcommand{\eqref}[1]{\mbox{(\ref{eq:#1})}} \usepackage{soul}
\newcommand{\setto}{\leftarrow} \newcommand{\defeq}{:=}
 \newcommand{\etc}{{\em
    etc.\/}\xspace} \newcommand{\ie}{{\em i.e.\/}\xspace}

\newtheorem{theorem}{Theorem}
\newcommand{\thmref}[1]{\mbox{Theorem~\ref{thm:#1}}}
\newcommand{\halm}{\hspace*{\fill} $\Box$\par}
\newenvironment{proof}{\noindent {\bf
    Proof.   }}{\halm\vspace{\baselineskip}}
\newenvironment{proof0}[1]{\noindent {\bf Proof of
    {#1}.   }}{\halm\vspace{\baselineskip}}

\newtheorem{lemma}{Lemma}
\newcommand{\lemref}[1]{\mbox{Lemma~\ref{lem:#1}}}
\newtheorem{cor}{Corollary}
\newcommand{\corref}[1]{\mbox{Corollary~\ref{cor:#1}}}
\newcommand{\itref}[1]{\mbox{\ref{it:#1}}}
\newtheorem{prop}{Proposition}
\newcommand{\propref}[1]{\mbox{Proposition~\ref{prop:#1}}}
\newcommand{\secref}[1]{\mbox{\S~\ref{sec:#1}}}
\newcommand{\tabref}[1]{\mbox{Table~\ref{tab:#1}}}
\newcommand{\appref}[1]{\mbox{Appendix~\ref{sec:#1}}}

\newcommand{\cip}{\mbox{$\perp\!\!\!\perp$}}
\newcommand{\indo}[2]{\mbox{$#1 \,\cip\, #2$}}

\usepackage{natbib}
\usepackage[colorlinks,citecolor=red,urlcolor=blue,bookmarks=false,hypertexnames=true]{hyperref}

\usepackage[titletoc,title]{appendix}

\title{Bounding Causes of Effects with Mediators}
\author{
\textbf
  {Philip Dawid}\thanks{University of Cambridge\quad
  \href{mailto:apd@statslab.cam.ac.uk}{apd@statslab.cam.ac.uk}} \and
  \textbf
{Macartan Humphreys}\thanks{Columbia University
  \& WZB Berlin\quad
  \href{mailto:mh2245@columbia.edu}{mh2245@columbia.edu}} \and
  \textbf
{Monica Musio}
  \thanks{Universit\`a degli Studi di Cagliari\quad
  \href{mailto:mmusio@unica.it}{mmusio@unica.it}}}
\date{\normalsize\today}

\begin{document}

\maketitle

\begin{abstract}
  \noindent Suppose $X$ and $Y$ are binary exposure and outcome
  variables, and we have full knowledge of the distribution of $Y$,
  given application of $X$.  From this we know the average causal
  effect of $X$ on $Y$.  We are now interested in assessing, for a
  case that was exposed and exhibited a positive outcome, whether it
  was the exposure that caused the outcome.  The relevant
  ``probability of causation", \pc, typically is not identified by the
  distribution of $Y$ given $X$, but bounds can be placed on it, and
  these bounds can be improved if we have further information about
  the causal process.  Here we consider cases where we know the
  probabilistic structure for a sequence of complete mediators between
  $X$ and $Y$.  We derive a general formula for calculating bounds on
  \pc for any pattern of data on the mediators (including the case
  with no data).  We show that the largest and smallest upper and
  lower bounds that can result from any complete mediation process can
  be obtained in processes with at most two steps.  We also consider
  homogeneous processes with many mediators.  $\pc$ can sometimes be
  identified as 0 with negative data, but it cannot be identified at 1
  even with positive data on an infinite set of mediators.  The
  results have implications for learning about causation from
  knowledge of general processes and of data on cases.
\end{abstract}


\section{Introduction}

Even the best possible evidence regarding the effects of a treatment
on an outcome is generally not enough to identify the probability that
the outcome was caused by the treatment.

For instance, researchers conducting randomised controlled trials may
determine that providing a medicine to school children increases the
overall probability of good health from one third to two thirds.  This
information, no matter how precise, is not enough to answer the
following question: Is Ann healthy because she took the medicine?  It
is not even enough to answer the question probabilistically.  The
reason is that, consistent with these results, it may be that the
medicine makes a positive change for 2 out of 3 students, but an
adverse change for the remainder: in that case the medicine certainly
helped Ann.  But it might alternatively be that the medicine makes a
positive change for 1 in 3 children but no change for the others.  In
that case the chances it helped Ann are just 1 in 2.  Of the children
taking the medicine, two thirds are healthy.  Half of these are
healthy because of the medicine, whereas the other half would have
been healthy anyway.

Put differently, the experimental data identifies the ``effects of
causes,'' (EoC) but we are interested in the reverse problem, of
quantifying ``causes of effects'' (CoE).  The CoE task of defining and
assessing the {\em probability of causation\/}
\citep{robins1989probability} in an individual case has been
considered by \citet{tian/pearl:probcaus,dawid2011role,
  yamamoto2012understanding,pearl2015causes,apd/mm/sef:ba,dawid2016bounding,dmm:lpr,murtas2017new}.
Note that this is distinct from the ``reverse causal question'' of
\citet{gelman2013ask}, which is an EoC task aimed at ascertaining
which causes have an effect on an outcome.

To understand causes of effects better, we might seek additional
evidence along causal pathways.  For example, researchers evaluating
development programs specify ``theories of change'' and seek evidence
for intermediate outcomes along a pathway linking treatment to
outcomes---most simply, Was the treatment received?  Was the medicine
ingested?  \citet{Van-Evera:1997} describes various tests that might
be implemented using such ancillary evidence.  A ``smoking gun test''
searches for evidence that, though unlikely to be found, would give
great confidence in a claim if it were to be found; a ``hoop test''
test is a search for evidence that we expect to find, but which, if
found to be absent, would provide compelling evidence against a
proposition (as if the proposition were asked to jump through a hoop).

Sometimes many points along a causal pathway are investigated.  An
intervention might be to provide citizens with information on
political corruption, in the hope that this will lead to ultimate
changes in politicians' behavior.  Researchers might then check many
points along a chain of intermediate outcomes.  Was the political
message delivered?  Was it understood? Was it believed? Did it induce
a change in behavior by citizens? Did this in turn produce a change in
behavior by politicians?

Seeing positive evidence at many points along a such a causal chain
would appear to give confidence that the final outcome is indeed
\textit{due} to the conjectured cause.  This is the core premise of
``process tracing,'' as deployed by qualitative political scientists
\citep{collier2011understanding}, as well as of mixed methods research
as used in development evaluation \citep{white2009theory}.  In the
most optimistic accounts it is assumed that, as one gets close enough
to a process, by observing more and more links in a chain, the link
between any two steps becomes less questionable and eventually the
causal process reveals itself \citep[581]{mahoney2012logic}.

We here provide a comprehensive treatment of the scope for inferences
of this form from knowledge of causal chains.  We obtain a general
formula for calculating bounds on the probability of causation, for an
arbitrary pattern of data along chains of binary variables.  We derive
implications of this formula, and calculate the largest and smallest
upper and lower bounds achievable from any causal chain consistent
with the known relation between $X$ and $Y$.  We give special
attention to what might appear to be the best possible conditions:
those in which causal processes really do follow a simple causal
chain, in which researchers have complete experimental evidence about
the probabilistic relationship between any two consecutive nodes in
the chain, in which the chain is arbitrarily long, in which the causal
effect of each intermediate variable on its successor climbs to 1, and
in which researchers observe outcomes consistent with positive effects
at every point on the chain.  We show that such information does
indeed increase confidence that an outcome can be attributed to a
cause and, for homogeneous chains at least, that the longer the chain
the better.  However, we find that even under these ideal conditions
our ability to narrow the bounds for the probability of causation can
be modest.  In the example of attributing Ann's health to good
medicine, a homogeneous process with arbitrarily many positive
intermediate steps observed might only tighten the bounds from
$[.5, 1]$ to $[.58, 1]$.

In contrast, we show that non-homogeneous processes can tighten the
bounds considerably.  For example, suppose Ann was prescribed the
medicine and recovered.  If we know that being prescribed the medicine
is the only way in which Ann could have obtained and taken the
medicine, and that taking the medicine helps anyone who would
otherwise be sick, then with positive evidence on a single
intermediate point on the causal chain---that Ann did indeed take the
medicine---we can identify the probability that prescribing the
medicine caused Ann's recovery at $2/3$.  (We are still short of 1,
because it is possible that Ann would have recovered even without the
medicine.)  A process like this, in which we observe a ``necessary
condition for a sufficient condition'', provides the largest possible
lower bound on the probability of causation available from any
observations on any chain.  At this point we have done the best
possible and more data along the chain will not help.

Although achieving identification of the probability of causation at 1
is generally elusive, negative data can yield identification at 0,
either in two steps from a heterogeneous process, or from alternating
data along an infinite homogeneous chain.  In this sense, information
on mediators can support ``hoop'' tests but not ``smoking gun'' tests.


\subsection{Plan of paper}
\label{sec:plan}

Existing results \citep{dawid2016bounding} have considered the case of
a single unobserved mediator.  We generalize this in two ways. First,
we consider situations with chains of arbitrary length.  Secondly, we
calculate bounds for general data, that is, for situations in which
the values of none, some or all the mediators are observed.

We proceed as follows.  Section~\ref{sec:prelim} introduces the
set-up, and provides general formulae for bounding the probability of
causation for a simple one-step process.  In \secref{med} we extend
these results to cases in which we know the structure of a complete
mediation process.  We consider various degrees of knowledge of the
values of the mediators for the individual case at hand: all
unobserved, all observed, or just some observed.  Our main result is
\thmref{prodpc}, which provides a general formula applicable to all
cases.

Section~\ref{sec:imp} draws out the detailed implications of this
result in a variety of contexts.  In \secref{nonhom} we investigate
the largest achievable lower and upper bounds from any sequence, and
find that these can be achieved by heterogeneous two-step processes.
Section~\ref{sec:hom} examines the case of homogeneous processes of
arbitrary length.  We show that an alternating pattern for the values
at all intermediate points can lead to a limiting value of 0 for the
probability of causation.  However, it is not generally possible for
even the most positive evidence to identify the probability of
causation---and {\em a fortiori\/} not possible to identify it at
1---even in the limit of infinitely many steps.
Section~\secref{strat} considers implications of our results for
gathering data on mediators.  In \secref{comp} we compare the bounds
based on knowledge of mediator processes with those achievable from
knowledge of covariates, which can be much tighter.  We summarise our
findings in \secref{conc}.  Various technical details for the proofs
in the paper are elaborated in three appendices.

\section{Preliminaries}
\label{sec:prelim}
We consider a binary treatment variable $X$ and binary outcome
variable $Y$.  We suppose we have access to experimental (or
unconfounded observational) data supplying values for
$\Pr(Y=y \mid X\setto x)$, where we use the notation $X\setto x$ to
denote a regime in which $X$ is set to value $x$ by external
intervention.

Define
\begin{eqnarray*}
  \tau &\defeq& \Pr(Y=1\mid X\setto 1) - \Pr(Y=1\mid X\setto 0)\\
  \rho &\defeq& \Pr(Y=1\mid X\setto 1) - \Pr(Y=0\mid X\setto 0).    
\end{eqnarray*}
Then $\tau$ is the {\em average causal effect\/} of $X$ on $Y$, while
$\rho$ is a measure of how common $Y=1$ is.

The transition matrix from $X$ to $Y$ (where the row and column labels
of any such matrix are implicitly $0$ and $1$ in that order) can be
written:
\begin{equation}
  \label{eq:P}
  P = P(\tau,\rho) \defeq\left(
    \begin{array}[c]{cc}
      \half(1+\tau-\rho) & \half(1-\tau+\rho)\\
      \half(1-\tau-\rho) & \half(1+\tau+\rho)
    \end{array}
  \right).
\end{equation}
All entries of $P$ must be non-negative: this holds if and only if
\begin{equation}
  \label{eq:rhotau}
  |\rho| + |\tau| \leq 1.    
\end{equation}
We have equality in \eqref{rhotau} if and only if one of the entries
of \eqref{P} is 1, in which case we term $P$ {\em degenerate\/}.  For
$\tau\geq 0$, this will happen if either $\rho = 1-\tau$, in which
case $\Pr(Y=1\mid X=1)=1$ and $X=1$ can be thought of as a sufficient
condition for $Y=1$; or $\rho = \tau-1$, in which case
$\Pr(Y=1\mid X=0)=0$, and $X=1$ can be thought of as a necessary
condition for $Y=1$.  Defining, for $ \tau\geq 0$,
\begin{equation}
  \label{eq:sigma}
  \sigma :=\frac{\rho}{1-\tau},
\end{equation}
we might thus regard $\sigma\in[-1,1]$ as measuring the
\textit{relative sufficiency} of $X=1$ for $Y=1$.\footnote{Although we
  do not focus on it, for $\tau< 0$ the analogous quantity
  $\frac{-\rho}{1+\tau}$ can be interpreted as the relative
  sufficiency of $X=1$ for $Y=0$.}



\subsection{Potential outcomes and causes of effects}
\label{sec:prcaus}

While knowledge of the transition matrix $P$, and in particular the
``average causal effect'' $\tau$, is directly relevant for EoC
(``effects of causes'') analysis, it is not enough to support CoE
(``causes of effects'') analysis.  For this we need to introduce the
pair of {\em potential outcomes\/}, $\bY = (Y_0,Y_1)$, where we
conceive of $Y_x$ as the value $Y$ would take, if $X\leftarrow x$.  We
regard both $Y_0$ and $Y_1$ as existing simultaneously, even prior to
setting the value of $X$, and as having a bivariate probability
distribution.

We can now define the following events in terms of $\bY$ (where
$\overline x$ denotes $1-x$, the value distinct from $x$, \etc):
\begin{description}
\item[General causation] $C^{(X,Y)}$ := ``$Y_1 \neq Y_0$''.
  
  That is, changing the value of $X$ will result in a change to the
  value of $Y$.  We can also describe this as ``$X$ affects $Y$.''

  When the relevant variables $X$ and $Y$ are clear from the context
  we will simplify the notation to $C$.


\item[Specific causation] $C^{(X,Y)}_{xy}$ :=
  ``$Y_x = y, Y_{\overline x} = \overline y$'' (for $x, y = 0$ or
  $1$).

  That is, changing the value of $X$ from $x$ to $\overline x$ would
  change the value of $Y$ from $y$ to $\overline y$.  We can also
  describe this as ``$X=x$ causes $Y=y$."  When the relevant variables
  $X$ and $Y$ are clear from the context we will simplify the notation
  to $C_{xy}$.

  We note that $C_{xy} = C_{\overline x \overline y}$.

\item[Probability of Causation.]

  In cases of interest we will have observed $X=x, Y=y$, and want to
  know {\em the probability that $X$ caused $Y$\/}, given this
  information.  We 
  denote this quantity by $\pc_{xy}^{(X,Y)}$, or $\pc_{xy}$ when the
  relevant variables $X$ and $Y$ are clear from the context.  Thus
  \begin{equation}
    \label{eq:pcxy}
    \pc_{xy} =  \Pr(C \mid X=x, Y=y) = \Pr(C_{xy} \mid X=x, Y_x=y).
  \end{equation}    
\end{description}

The joint distribution for $\bY$, while constrained by knowledge of
the transition matrix $P$, is in general not fully determined by it.
Rather, we can only deduce that it has the form of \tabref{y0y1},
where the marginal probabilities agree with \eqref{P} according to
$\Pr(Y_x = y) = \Pr(Y = y \mid X \setto x)$.
\begin{table}[h] \centering
  \begin{tabular}[c]{c|cc|c} \multicolumn{1}{c}{}& $Y_1=0$ & \multicolumn{1}{c}{$Y_1=1$}\\\hline
    $Y_0 = 0$ & $\half(1-\rho - \xi)$ & $\half(\xi+\tau)$ & $\half(1+\tau-\rho)$\\
    $Y_0 = 1$ & $\half (\xi-\tau)$ & $\half(1+\rho-\xi)$ & $\half(1-\tau+\rho)$ \\\hline
                                                 &  $\half(1-\tau-\rho)$ & $\half(1+\tau+\rho)$ & 1
  \end{tabular}
  \caption{$\Pr(Y_0=y_0, Y_1=y_1)$}
  \label{tab:y0y1}
\end{table}
However, the internal entries of \tabref{y0y1} are not determined by
$P$, but have one degree of freedom, expressed by the ``slack''
quantity $\xi$ = $\xi(P)$.  We see that
\begin{equation}
  \label{eq:Cineq}
  \xi = \Pr(Y_0 = 0, Y_1=1) + \Pr(Y_0 = 1, Y_1=0)  = \Pr(C),
\end{equation}
the probability of general causation.

The only constraints on $\xi$ are that all internal entries of
\tabref{y0y1} must be non-negative, which holds if and only if
\begin{equation}
  \label{eq:basicineq}
  |\tau| \leq \xi \leq  1-|\rho|.
\end{equation}
In particular $\xi$, and thus the bivariate distribution of
$(Y_0, Y_1)$ in \tabref{y0y1}, is uniquely determined by $P$ if and
only $P$ is degenerate.

We further note
\begin{eqnarray}
  \label{eq:c00}
  \Pr(C_{00}) =   \Pr(C_{11}) &=& \half(\xi+\tau)\\
  \label{eq:c01}
  \Pr(C_{01}) =   \Pr(C_{10}) &=& \half(\xi-\tau)
\end{eqnarray}
whence, by \eqref{basicineq},
\begin{eqnarray}
  \label{eq:00ineq}
  \max\{0, \tau\} \,\,\leq \,\,\Pr(C_{00})=   \Pr(C_{11}) &\leq& \half(1 + \tau -  |\rho|)\\
  \label{eq:01ineq}
  \max\{0, -\tau\} \,\,\leq\,\, \Pr(C_{01}) =   \Pr(C_{10}) &\leq& \half(1- \tau -  |\rho|).
\end{eqnarray}


Throughout this article we shall assume {\em no confounding\/},
expressed mathematically as $\indo X \bY$.  Then
\begin{eqnarray*}
  \pc_{xy} &=& \frac{\Pr(C_{xy})}{\Pr(Y_x=y)}\\
           &=& \frac{\Pr(C_{xy})}{\Pr(Y=y \mid X\setto x)}
\end{eqnarray*}
which is thus subject to the interval bounds, given by \eqref{00ineq}
or \eqref{01ineq}, as appropriate, divided by the known entry
$\Pr(Y=y \mid X\setto x)$ of the transition matrix $P$.

This analysis delivers the following lower and upper bounds (prefix
``{\it s}'' for ``simple''):
\begin{eqnarray}
  \label{eq:simpbound00}
  \slb_{00} :=
  \frac{\max\{0,\tau\}}{\Pr(Y=0 \mid X \setto 0)}\,\,\, \leq \,\,\, 
  \pc_{00} &\leq&  \frac{\half(\tau + 1-|\rho|)}{\Pr(Y=0 \mid X \setto 0)}
                  =: \sub_{00}
  \\
  \label{eq:simpbound01}
  \slb_{10} :=
  \frac{\max\{0,-\tau\}}{\Pr(Y=0 \mid X \setto 1)}\,\,\, \leq \,\,\,   
  \pc_{10} &\leq& \frac{\half(1- |\rho| - \tau)}{\Pr(Y= 0\mid X \setto 1)}
                  =: \sub_{10}
  \\
  \label{eq:simpbound10}
  \slb_{01} := 
  \frac{\max\{0,-\tau\}}{\Pr(Y=1 \mid X \setto 0)}\,\,\, 
  \leq \,\,\, \pc_{01} &\leq&  \frac{\half(1-|\rho|-\tau)}{\Pr(Y=1 \mid X \setto 0)}
                              =: \sub_{01}
  \\
  \label{eq:simpbound11}
  \slb_{11} := 
  \frac{\max\{0,\tau\}}{\Pr(Y=1 \mid X \setto 1)}\,\,\, \leq \,\,\, 
  \pc_{11}  &\leq& \frac{\half(\tau +1 -|\rho|)}{\Pr(Y=1 \mid X \setto 1)}
                   =: \sub_{11}.
\end{eqnarray}
In the absence of additional information, the above bounds constitute
the best available inference regarding the probability of causation.

Specifically, when $\tau\geq 0$, on defining
\begin{eqnarray}
  \label{eq:gamma}
  \gamma &:=& \frac{1-\tau-|\rho|}{1-\tau+|\rho|}=\frac{1-|\sigma|}{1+|\sigma|} \\
  \label{eq:delta}
  \delta &:=& \frac{1+\tau-|\rho|}{1+\tau+|\rho|}
\end{eqnarray}
we have the following upper bounds:

For $\rho \geq 0$:
\begin{eqnarray}
  \label{eq:s00++}
  \sub_{00} &=& 1\\
  \label{eq:s01++}
  \sub_{01} &=& \gamma\\
  \label{eq:s10++}
  \sub_{10} &=& 1\\
  \label{eq:s11++}
  \sub_{11} &=& \delta
\end{eqnarray}

For $\rho < 0$:
\begin{eqnarray}
  \label{eq:s00+-}
  \sub_{00} &=& \delta\\
  \label{eq:s01+-}
  \sub_{01} &=& 1\\
  \label{eq:s10+-}
  \sub_{10} &=& \gamma\\
  \label{eq:s11+-}
  \sub_{11} &=& 1
\end{eqnarray}



\subsection{Special case}
\label{sec:specsimp}
A particular interest is in cases where $\tau>0$ (so the overall
effect of $X$ and $Y$ is positive) and we observe positive outcomes,
$X=1$, $Y=1$.  In this case we omit the subscript $11$.  We have
\begin{equation}
  \label{eq:now}
  \pc = \frac{\xi+\tau}{2\Pr(Y=1 \mid X \setto 1)},
\end{equation}
and interval bounds given by
\begin{equation}
  \label{eq:simple}
  \slb = \frac{2\tau}{1+\tau+\rho} \leq \pc \leq \sub = 
  \left\{
    \begin{array}{ll}
      \delta& (\rho\geq 0)\\
      1 & (\rho < 0)
    \end{array}
  \right.
\end{equation}
This result agrees with
\citep{tian/pearl:probcaus,dawid2011role,dmm:lpr}.

$\pc$ is identified (\ie, the interval in \eqref{simple} reduces to a
single point) if and only if $|\rho| = 1-\tau$, which holds when $P$
is degenerate with either the lower left or upper right element of $P$
being 0.  In the former case $\pc=\tau$, while in the latter case
$\pc = 1$.

More generally, we have
$\slb = \tau/\Pr(Y=1 \mid X \setto 1) \geq \tau$, so $\pc \geq \tau$.

\section{Bounds from mediation}
\label{sec:med}

We now suppose that, in addition to $X$ and $Y$, we can gather data on
one or more binary mediator variables $M_1,\ldots,M_{n-1}$.  We also
define $M_0 \equiv X$ and $M_n \equiv Y$.  We are interested in
assessing the probability that $X=x$ caused $Y=y$ for a new case where
we have information on the values of some or all of the mediators
$M_1, \dots, M_{n-1}$.

We assume that the data are based on experiments, or in any case are
such as to allow us to determine the one-step interventional
probabilities $\Pr(M_{i+1} = m_{i+1} \mid M_i \setto m_i)$,
$i=0,\ldots,n-1$.  We shall here confine attention to the case of a
{\em complete mediation sequence\/}, where
$$\Pr(M_{i+1} = m_{i+1} \mid M_j \setto m_j, j=0,\ldots,i) = \Pr(M_{i+1} = m_{i+1} \mid M_i \setto m_i), \quad(i=0,\ldots,n-1).$$
We shall further suppose that, for any new case considered, there is
no confounding at every step, so that
$$\Pr(M_{i+1} = m_{i+1} \mid M_j = m_j, j=0,\ldots,i) = \Pr(M_{i+1} = m_{i+1} \mid M_i \setto m_i), \quad(i=0,\ldots,n-1).$$
In this case the sequence of observations
$(X\equiv M_0,\ldots,M_n\equiv Y)$ on a new case will form a
(generally non-stationary) Markov chain.  This is an empirically
testable consequence of our assumptions, assumptions which would
therefore be falsified if the Markov property is found to fail
(although those assumptions are not guaranteed to be valid when it is
found to hold.)


Let the transition matrix from $M_{i-1}$ to $M_{i}$ be
$P_i = P(\tau_i,\rho_i)$, and the overall transition matrix from $X$
to $Y$ be $P = P(\tau,\rho)$.  We shall write
\begin{equation}
  \label{eq:decomp}
  P =P_1 \mid P_2 \ldots \mid P_n
\end{equation}
to indicate that we are assuming the above mediation sequence, and
refer to \eqref{decomp} as a {\em decomposition\/} of the matrix $P$.
In particular we then have $P = P^{(n)} \defeq \prod_{i = 1}^n P_{i}$.


{W}e can readily
show 
by induction that
\begin{eqnarray}
  \label{eq:taun} 
  \tau = \tau^{(n)}  &\defeq& \prod_{i=1}^{n} \tau_i\\
                                                \label{eq:rhon} 
                                                \rho = \rho^{(n)} &\defeq & \sum_{i=1}^{n} \rho_i\prod_{j=i+1}^{n}\tau_j.
\end{eqnarray}
 
In particular, for the case $n=2$, \eqref{rhon} becomes
\begin{equation}
  \label{eq:rho2}
  \rho = \rho_1 \tau_2 + \rho_2.
\end{equation}

On account of \eqref{taun} we have the following result:
\begin{theorem}
  \label{thm:prod}
  The average causal effect of $X$ on $Y$ is the product of the
  successive average causal effects of each variable in the sequence
  on the following one.
\end{theorem}

Again, to conduct CoE rather than EoC analysis, we introduce, for
$i\geq 1$, bivariate variables
$$\bM_i \defeq (M_{i0}, M_{i1})$$
where $M_{im}$ denotes the potential value of $M_i$ under
$M_{i-1} \leftarrow m$, supposed unaffected by values of previous
$M$'s.  We further assume that the variable $\bM_i$ is common to all
the various worlds, whether actual or counterfactual, under
consideration.  The actually realised values $(M_i)$ satisfy
$M_i = M_{i,{M_{i-1}}}$.

As the expression of our ``no confounding'' assumptions, we impose
mutual independence between $X$, $\bM_1$,\ldots,$\bM_n$.

\begin{theorem}
  \label{thm:medaffect}
  $C^{(X,Y)} = \bigcap_{i=0}^{n-1}\,C^{(M_i,M_{1+1})}$.  That is to
  say, $M_0 \equiv X$ affects $M_n \equiv Y$ if and only if each $M_i$
  affects the next.
\end{theorem}

\begin{proof}
  Suppose first that each variable affects the next.  Then changing
  the value of $X$ will change that of $M_1$, which in turn will
  change that of $M_2$, and so on until the value of $Y$ is changed,
  so showing that $X$ affects $Y$.  Conversely, if, for some $j<n$,
  $M_j$ does not affect $M_{j+1}$, then, whether or not $M_j$ has been
  changed, the value of $M_{j+1}$ will be unchanged, whence so too
  will that of $M_{j+2}$, and so on until the value of $Y$ is
  unchanged, whence $X$ does not affect $Y$.
\end{proof}

\renewcommand{\theenumi}{(\roman{enumi})}

\begin{cor}
  \label{cor:xiprod}
  \quad\\
  \begin{enumerate}
  \item \label{it:xi1}
    $\Pr(C^{(X,Y)}) = \prod_{i=1}^{n}\,\Pr(C^{(M_{i-1},M_{1})})$
  \item \label{it:xi2} $\xi(P) = \prod_{i=1}^{n}\,\xi(P_i)$
  \item \label{it:xi3} Given the detailed information on the
    decomposition \eqref{decomp}, the constraints on $\xi = \xi(P)$
    are now:
    \begin{equation}
      \label{eq:xidecomp}
      |\tau| \leq \xi \leq \prod_{i=1}^{n}\left(1-|\rho_i|\right).
    \end{equation}
  \end{enumerate}

\end{cor}
\begin{proof}
  \begin{description}
  \item[\itref{xi1}] By the assumed mutual independence of the
    $(\bM_i)$.
  \item[\itref{xi2}] By \eqref{Cineq}.
  \item[\itref{xi3}] By \itref{xi2}, \eqref{basicineq} for each $P_i$,
    and \eqref{taun}.
  \end{description}
\end{proof}

On account of \itref{xi1} we have:
\begin{cor}
  \label{cor:genprod}
  For any decomposition, the probability that $X$ affects $Y$ is the
  product of the probabilities that each variable in the sequence from
  $X$ to $Y$ affects the next in the sequence.
\end{cor}

On comparing \eqref{xidecomp} with \eqref{basicineq}, we see that
detailed knowledge of the mediation process has not changed the lower
bound for $\xi$.  However, the upper bound is typically reduced:
\begin{theorem}
  \label{thm:ubless}
  The upper bound of \eqref{xidecomp}, which takes into account the
  decomposition \eqref{decomp}, does not exceed the upper bound of
  \eqref{basicineq}, which ignores the decomposition.  It will be
  strictly less if all the $P_i$ are non-degenerate with
  $\rho_i \neq 0$.
\end{theorem}

\begin{proof}
  Consider first the case $n=2$.  Then
  \begin{eqnarray}
    \nonumber
    |\rho| &=& |\rho_1 \tau_2 + \rho_2|\quad\mbox{by \eqref{rho2}}\\
    \label{eq:z1} 
           &\leq& |\rho_1| |\tau_2| + |\rho_2|\\
    \label{eq:z2} 
           &\leq& |\rho_1| (1 - |\rho_2|) + |\rho_2| \quad\mbox{by \eqref{rhotau}.}
  \end{eqnarray}
  It follows that
  \begin{equation}
    \label{eq:mod2}
    (1-|\rho_1|)(1-|\rho_2|) \leq  1-|\rho|.
  \end{equation}
  Moreover, we shall have strict inequality in \eqref{z2}, and hence
  also in \eqref{mod2}, if $P_2$ is non-degenerate and
  $\rho_1 \neq 0$.

  The result for general $n$ follows easily by induction.
\end{proof}
We note that the above condition for strict inequality in
\eqref{mod2}, while sufficient, is not necessary.  For example, in the
case $n=2$ it will also hold if $\rho_1\tau_2$ and $\rho_2$ have
different signs, since then we would have strict inequality in
\eqref{z1}.

It follows from \eqref{xidecomp} and \eqref{mod2} that collapsing two
mediators into a single one can only increase the upper bound for
$\xi$:
\begin{cor}
  \label{cor:ubusmall}
  Consider two decompositions $ P =P_1 \mid P_2 \ldots \mid P_n$ and
  $ P =P_1 \mid \ldots \mid P_i \mid Q \mid P_{i+2} \mid \ldots \mid
  P_n$, where $Q = P_i P_{i+1}$.  Then the upper bound for $\xi$ for
  the former does not exceed that for the latter.
\end{cor}

\subsection{Bounds when mediators are unobserved}

Suppose first that, for the new case, we have observed $X=x, Y=y$, but
the values of the mediators are not observed.  Even in this case, as
shown for the two-term decomposition in \citet{dawid2016bounding},
knowledge of the decomposition \eqref{decomp} of $P$ can alter the
bounds for $\pc$.

Indeed, in this case \eqref{pcxy} still applies, where $\Pr(C_{xy})$
is given by \eqref{c00} or \eqref{c01} as appropriate, but now with
$\xi$ subject to the revised bounds of \eqref{xidecomp}.  In each case
the lower bound is unaffected, but, by \thmref{ubless}, the upper
bound is reduced.

This analysis delivers the following revised bounds (prefix ``{\it
  u}'' for ``unobserved mediators''):
\begin{eqnarray}
  \label{eq:ubound00}
  \ulb_{00} := \slb_{00} =
  \frac{\max\{0,\tau\}}{\Pr(Y=0 \mid X \setto 0)}\,\,\, \leq \,\,\, 
  \pc_{00} &\leq&  \frac{\tau + \prod_{i=1}^n \left(1-|\rho_i|\right)}
                  {2\Pr(Y=0 \mid X \setto 0)}
 %
                                                =: \uub_{00}
  \\
  \label{eq:ubound10}
  \ulb_{10} := \slb_{10} = 
  \frac{\max\{0,-\tau\}}{\Pr(Y=0 \mid X \setto 1)}\,\,\, \leq \,\,\,   
  \pc_{10} &\leq& \frac{\prod_{i=1}^n \left(1-|\rho_i|\right)- \tau}{2\Pr(Y= 0\mid X \setto 1)}
                    =: \uub_{10}
  \\
  \label{eq:ubound01}
  \ulb_{01} := \slb_{01} = 
  \frac{\max\{0,-\tau\}}{\Pr(Y=1 \mid X \setto 0)}\,\,\, 
  \leq \,\,\, \pc_{01} &\leq&  \frac{\prod_{i=1}^n \left(1-|\rho_i|\right)-\tau}{2\Pr(Y=1 \mid X \setto 0)}
                                       =: \uub_{01}
  \\
  \label{eq:ubound11}
  \ulb_{11} := \slb_{11} = 
  \frac{\max\{0,\tau\}}{\Pr(Y=1 \mid X \setto 1)}\,\,\, \leq \,\,\, 
  \pc_{11}  &\leq& \frac{\tau +\prod_{i=1}^n \left(1-|\rho_i|\right)}{2\Pr(Y=1 \mid X \setto 1)}
                    =: \uub_{11}
\end{eqnarray}

\subsection{Special case}
\label{sec:specunob}

In particular, for the case $\tau>0$, where we observe $X=1$, $Y=1$
(but the values of mediators are not observed), we have revised bounds
\begin{equation}
  \label{eq:simpleu}
  \ulb := \frac{2\tau}{1+\tau+\rho} \leq \pc \leq \frac{\tau + \prod_{i=1}^n \left(1-|\rho_i|\right)}{1+\tau+\rho} =:\uub.
\end{equation}
For $n=2$ this agrees with the analysis of \citet{dawid2016bounding}.

\subsection{Bounds when some or all mediators are observed}

Now suppose that, in addition to $X=x$, $Y=y$, we also observe data on
$k$ mediators ($0\leq k \leq n-1$) for the new case.  In particular we
observe $M_{i_r} = m_{i_r}$, for
$0 < i_1 < \ldots i_r \ldots < i_k < n$.  For notational simplicity we
write $\tM_r$ for $M_{i_r}$, $\tm_r$ for $m_{i_r}$.  We also identify
$\tM_0 \equiv X$ and $\tM_{k+1} \equiv Y$ (so $\tm_0 = x$,
$\tm_{k+1}=y$).

The relevant probability of causation is now
\begin{displaymath}
  \tpc_{xy} := \Pr\left(C \mid  \tM_r =
    \tm_r,\,\,i=0,\ldots,k+1\right).
\end{displaymath}

Note that in contrast to the difference between
\eqref{ubound00}--\eqref{ubound11} on the one hand and
\eqref{simpbound00}--\eqref{simpbound11} on the other hand, which
relate to the same quantity $\pc_{xy}$ but express different
conclusions about it, $\tpc_{xy}$ is a genuinely different quantity
from $\pc_{xy}$, since it conditions on different information about
the new case.

\begin{theorem}
  \label{thm:prodpc}
  Given observations on $X, \tM_1,\ldots,\tM_k, Y$, the probability
  that $X$ caused $Y$ is given by the product of the probabilities
  that each observed term in the sequence caused the next observed
  term:
  \begin{displaymath}
    \tpc_{xy}  = \prod_{r=0}^k \pc^{(\tM_r,\tM_{r+1})}_{\tm_r \,\tm_{r+1}}.   
  \end{displaymath}

\end{theorem}

\begin{proof}
  From \thmref{medaffect} we have
  \begin{displaymath}
    C = \bigcap_{r=0}^{k}\,C^{(\tM_r,\tM_{r+1})},
  \end{displaymath}
  whence, using the ``no-confounding'' independence properties,
  \begin{eqnarray}
    \nonumber  
    \tpc_{xy} &=& \prod_{r=0}^k\Pr\left(C^{(\tM_r,\tM_{r+1})}\,\right | \left.   \tM_r = \tm_r, \tM_{r+1} =  \tm_{r+1}\right)\\
    \label{eq:prodtilde}                    &=&  \prod_{r=0}^k \pc^{(\tM_r,\tM_{r+1})}_{\tm_r \,\tm_{r+1}}.   
  \end{eqnarray}
\end{proof}

Now since we have the decomposition information about the mediators
(if any) occurring between $\tM_r \equiv M_{i_r}$ and
$\tM_{r+1} \equiv M_{i_{r+1}}$, but not their values for the new case,
the bounds on any
factor 
in \eqref{prodtilde} will, {\em mutatis mutandis\/}, have the form of
the relevant expressions for $\ulb_{xy}$ and $\uub_{xy}$, as displayed
in \eqref{ubound00}---\eqref{ubound11}.  Then the overall lower
[resp., upper] bound on $\tpc_{xy}$ will be the product of these lower
[resp., upper] bounds, across all terms.  This procedure supplies a
complete recipe for determining the appropriate bounds on $\tpc_{xy}$
in the knowledge of the full decomposition of $P$ and the values of
the observed mediators for the new case.

\subsection{Special cases}
\label{sec:specobs}
Again consider the case $\tau > 0$, $X=Y=1$.  On account of
\eqref{taun} we can, after possibly switching the labels $0$ and $1$
for some of the $M_i$'s, take $\tau_i > 0$, all $i$.  We assume
henceforth that this is the case.  The above procedure then delivers
lower bound $0$ unless $\tm_i = \tm_{i-1}$, all $i$, so that
$m_i = 1$, all $i$.  In that case we obtain lower bound (with prefix
``{\it o}'' for ``observed mediators''):
\begin{eqnarray}
  \nonumber \olb &:=& \frac{\tau}{\prod_{r=0}^k \Pr\left(\tM_{r+1} = 1 \mid \tM_{r} = 1\right)}\\
  \label{eq:lbjoint}
                 &=& \frac{\tau}{\Pr\left(Y=1, \tM_r = \tm_r, r=2,\ldots,k \mid X=1\right)}.
\end{eqnarray}
It is easy to see that this lower bound can only increase if we
introduce further observed mediators.  It follows that the smallest
lower bound occurs when the are no observed mediators, when it reduces
to $\ulb=\slb$ as in \eqref{simpleu} and \eqref{simple}; while the
largest lower bound occurs when all mediators are observed (all taking
value 1)---that is to say, there is positive evidence for every link
in the mediation chain.

In the remainder of this paper we shall give special attention to this
case, and write simply $\tpc$ for $\tpc_{11}$, \etc\@ The bounds for
$\tpc$ are then:
\begin{equation}
  \label{eq:prodbounds}
  \olb := \prod_{i=1}^n\left(\frac{2\tau_i}{1+\tau_i+\rho_i}\right)\,\,\, 
  \leq \,\,\, \tpc \,\,\,
  \leq \,\,\,  \prod_{i=1}^n\left(\frac{1+\tau_i-|\rho_i|}{1+\tau_i+\rho_i}\right) =: \oub .
\end{equation}




The following result follows directly from the above considerations:
\begin{lemma}
  \label{lem:lbineq}
  The lower bound $\olb$ of \eqref{prodbounds} is at least as large as
  the lower bound $\slb$ of \eqref{simple}.
\end{lemma}

It is not, however, always the case that $\oub \leq \sub$: see
\eqref{C} below.

\section{Implications}
\label{sec:imp}
Equation \eqref{prodtilde} provides a general formula for calculating
bounds on the probability of causation for any pattern of data
observed on mediating variables (including no data).

We now derive implications from this analysis.

\subsection{Largest and smallest upper and lower bounds}
\label{sec:nonhom}

Consider an arbitrary decomposition of $P$:
\begin{equation}
  \label{eq:arb}
  P = P_1 \mid P_2 \mid \ldots \mid P_n,
\end{equation}
with $P=P(\tau,\rho)$, $P_i=P(\tau_i,\rho_i)$.  We restrict attention
to the case $\tau >0$ and assume that variables are labeled so that
each $\tau_i > 0$.

We investigate the smallest and largest achievable values for
$\ulb, \uub, \olb, \oub, \nlb$, $\nub$ (prefix $m$ for mixed evidence)
and show that in each case these are achievable by decompositions
involving at most one mediator.

\begin{theorem}\label{thm:main}
  Let the (known, fixed) transition matrix from $X$ to $Y$ be
  $P = P(\tau, \rho)$, with $\tau>0$ and $|\rho| < 1-\tau$.  The
  largest and smallest upper and lower bounds from any complete
  mediation process for the case with mediators unobserved, for the
  case with positive outcomes on all mediators observed, and for mixed
  cases, that include some negative evidence on the mediators, are as
  given in \tabref{resultstable}.
  \begin{table}[btp]
    \centering
    \begin{tabular}{c|c|l|l|l}
      \hline 
      &	&  No evidence &   Positive evidence & Mixed evidence\\  \hline 
      Largest	& Upper &$\overline{\uub} = \frac{1+\tau-|\rho|}{1+\tau+\rho}$ & $\overline{\oub} = \min\{1, 1-\rho\}$  & $\overline{\nub} = 1$  \\ 
      & Lower & $\overline{\ulb} =\frac{2\tau}{1+\tau+\rho}$  & $\overline{\olb} = \frac{1+\tau-\rho}{2}$   & $\overline{\nlb} = 0$  \\ 	\hline 	
      Smallest	& Upper &$\underline{\uub} =\frac{2\tau}{1+\tau+\rho}$ (*) & $\underline{\oub} =\frac{2\tau}{1+\tau+\rho}$  (*) & $\underline{\nub} = 0$ (*)\\ 
      & Lower  & $\underline{\ulb} =\frac{2\tau}{1+\tau+\rho}$ & $\underline{\olb}=\frac{2\tau}{1+\tau+\rho}$& $\underline{\nlb} = 0$\\ 	\hline 	
    \end{tabular}
    \caption{Largest and smallest achievable upper and lower bounds from decompositions of any length, given no mediators observed, positive evidence observed for all mediators, or mixed evidence is observed.   (*) Indicates that \pc can be identified.}
    \label{tab:resultstable}
  \end{table}
  These can all be achieved by decompositions of length 1 or 2.
\end{theorem}

\begin{proof}
  See \appref{mainproof}.
\end{proof}

The largest upper bound with mediators unobserved, $\overline{\uub}$,
can be achieved without any mediators.  Since unobserved mediators do
not alter the lower bound we have
$\overline{\ulb} = \underline{\ulb} = {\slb}$.  In addition we have
$\underline{\uub} = {\slb}$, which is achievable, for example, from
the following decomposition:
\begin{equation}
  \label{eq:ba1}
  P = \left.
    \left(
      \begin{array}{cc}
        \frac{2\tau}{1+\tau+\rho} & \frac{1-\tau+\rho}{1+\tau+\rho} \\
        0 & 1
      \end{array}
    \right)\,
  \right|
  \left(
    \begin{array}[c]{cc}
      1 & 0\\
      \frac{1-\tau-\rho}2 &  \frac{1+\tau+\rho}2 
    \end{array}
  \right).
\end{equation}

Note that with this decomposition $\pc$ is identified {\em via\/} two
degenerate transition matrices: $X=1$ is a sufficient condition for
$M=1$, while $M=1$ is a necessary condition for $Y=1$.

The smallest upper and lower bounds available when mediators are
observed agree with the simple lower bound.  Positive evidence cannot
reduce the lower bound, but it can reduce the upper bound to the lower
bound, at which point $\tpc$ is identified.  This can be achieved by
the same decomposition given in \eqref{ba1}.

The largest upper bound with positive evidence on mediators,
$\overline{\oub}$, can exceed the simple upper bound when $\rho>0$.
It is achieved by the following two-term decomposition, involving a
single mediator:

\begin{equation}
  \label{eq:C}
  P = \left.
    \left(
      \begin{array}[c]{cc}
        \frac{1-\rho + \tau}{2(1-\rho)} & \frac{1 - \rho -  \tau}{2(1-\rho)} \\ 
        \frac{1 - \rho - \tau}{2(1-\rho)} & \frac{1-\rho + \tau}{2(1-\rho)}
      \end{array}
    \right)\,
  \right|
  \left(  
    \begin{array}{cc} 
      {1-\rho} & \rho 
      \\ 0 & 1	
    \end{array}
  \right).
\end{equation}


The lower bound can be raised with positive information on mediators,
and takes its largest value with the following degenerate two-term
decomposition $P = P_1 \mid P_2$, involving a single mediator:

\begin{equation}
  \label{eq:ab1}
  P = \left.
    \left(
      \begin{array}[c]{cc}
        1 & 0\\
        \frac{1-\tau-\rho}{1+\tau-\rho} & \frac{2\tau}{1+\tau-\rho}
      \end{array}
    \right)\,
  \right|
  \left( 
    \begin{array}{cc} 
      \frac{1+\tau-\rho}{2} & \frac{1-\tau+\rho}{2} \\
      0 & 1
    \end{array}
  \right).
\end{equation}

With this decomposition $\tpc$ is identified {\em via\/} two
degenerate transition matrices: in this case $X=1$ is a necessary
condition for $M=1$, while $M=1$ is a sufficient condition for $Y=1$.
The largest lower bound with positive evidence from this decomposition
is $\frac{1+\tau - \rho}{2}$ which can fall far short of 1, implying
that in general mediators cannot provide ``smoking gun'' evidence that
$X=1$ caused $Y=1$.

For the case with mixed evidence on the mediators the lower bound is
always 0.  The smallest upper bound is also 0, which can be achieved
by the decomposition \eqref{ab1} above, with the single mediator
observed at 0 (the key feature of this decomposition is that $Y=1$ can
not be caused by $M=0$).  In this case $\tpc$ is identified at 0,
showing that it is possible for negative data on mediators to provide
``hoop'' evidence that $X=1$ did not cause $Y=1$.  The highest upper
bound, $\nub = 1$, can be achieved by a two-step decomposition
$P(\tau,\rho) = P(\tau_1,\rho_1) \mid P(\tau_2,\rho_2)$, with the
mediator taking value 0.  For $\rho\leq 0$ this occurs with the
decomposition with parameters
\begin{equation}
  \label{eq:ext1}
  \tau_1 = \frac{2\tau}{1+\tau+\rho}\qquad \rho_1 = 0 \qquad\tau_2
  = \frac{1+\tau+\rho}{2} \qquad \rho_2 = \rho.
\end{equation}


For $\rho\geq 0$ it occurs with decomposition parameterized by
\begin{equation}
  \label{eq:ext2}
  \tau_1 = \frac{\tau(1+\rho+\tau)}{2(\tau + \rho)}\qquad \rho_1 =
  \frac{\rho(1+\rho+\tau)}{2(\tau + \rho)} \qquad\tau_2 =
  \frac{2(\tau+\rho)}{1+\tau+\rho} \qquad\rho_2 = 0.
\end{equation}

%

\subsection{Homogeneous transitions}
\label{sec:hom}
Throughout this section we confine attention to the special case
$\tau>0$, $X=Y=1$.  We specialize further to the case of a constant
one-step transition matrix, $P_i = P' =P(\tau',\rho')$ for all $i$.
We define $\sigma'$, $\gamma'$, $\delta'$ in terms of $\tau'$ and
$\rho'$ in parallel to \eqref{sigma}, \eqref{gamma} and \eqref{delta}.

In this case, by \eqref{taun} and \eqref{rhon}, we have
\begin{eqnarray}    \label{eq:tauhom} \tau  &=& (\tau')^n\\
  \label{eq:rhohom} \rho &=& \rho' \times \frac{1-(\tau')^n}{1-\tau'}  = \rho' \times \frac{1-{\tau}}{1-\tau'}.
\end{eqnarray}
In particular, we note that the relative sufficiency of $X$ for $Y$ is
preserved at each intermediate step:
$\sigma' = {\rho'}/(1-\tau') ={\rho}/(1-\tau) = \sigma$.  It follows
that $\gamma'=\gamma$.

We have
\begin{eqnarray}
  \label{eq:tau}
  \tau' &=& \tau^{1/n}\\
  \label{eq:rho}
  \rho' &=& \rho \times \frac{1-{\tau}^{1/n}}{1-\tau}.
\end{eqnarray}
Note that, for large $n$, $\tau'$ must be close to 1 and $\rho'$ close
to 0, with the same sign as $\rho$.


Using \eqref{tau} and \eqref{rho} in \eqref{simpleu} and
\eqref{prodbounds} yield the following bounds for a homogeneous
process:
\begin{eqnarray}
  \label{eq:pocxbounds}
  \ulb_n &=& \frac{2\tau}{1+\tau+\rho} \\
  \uub_n &=& \frac{\tau +  \left(1-|\rho'|\right)^n}{1+\tau+\rho}\\
  \label{eq:pcbounds2}
  \olb_n &=&  \tau\left(\frac {2}{1+\tau'+\rho'}\right)^n\\
  \label{eq:pcbounds3}
  \oub_n &=& 
             \left\{
             \begin{array}{ll}
               (\delta')^n& (\rho\geq 0)\\
               1 & (\rho < 0).
             \end{array}
                   \right.
    %
\end{eqnarray}
In particular, for the degenerate cases $|\rho| = 1-\tau$, so that
$|\rho'| = 1-\tau'$, we see, that for all $n$,
$\pc$ and $\tpc$ are both identified, at $\tau$ when $\rho=1-\tau$,
and at 1
when $\rho= \tau-1$---the existence of the mediators being irrelevant
in these cases.

\paragraph{Mixed evidence}
Here we assume the process is non-degenerate.

For the case with some negative evidence the lower bound, $\nlb_n$
say, is always 0, as noted in Section \secref{specobs}.  The upper
bound, however, depends on the particular pattern of positive and
negative evidence.  For any sequence $s$ of observations on
consecutive mediators (allowing $M_0 \equiv X$ and $M_n \equiv Y$,
both required to take value 1), denote the associated upper bound by
$\ub(s)$.  Let $\bs$ denote a full sequence of observations (\ie, on
all $n+1$ mediators).  We search for a full sequence $\bs_0$ yielding
the maximum value, $\nub_n$ say, of $\ub(\bs)$.

\begin{theorem}
  \label{thm:mixed}
  For large enough $n$, we have
  $$\nub_n = \left\{
    \begin{array}[c]{ll}
      \gamma^{n/2} & (n \mbox{ even})\\
      \gamma^{(n-1)/2}\,\delta' & (n \mbox{ odd})
    \end{array}
  \right.$$ The optimal sequence $\bs_0$ alternates $10101\ldots$,
  except, if $n$ is odd, for the final 2 symbols.
\end{theorem}

\begin{proof}
  See \appref{mixproofs}.
\end{proof}
  
For $\rho = 0$ the smallest possible upper bound is $1$ for all $n$.
Otherwise, $\nub_n\rightarrow 0$ as $n\rightarrow\infty$.  Then with
alternating evidence on many mediators the associated probability of
causation, $\hpc$ say, is effectively identified as $0$.

Figure~\ref{fig:withn} plots the intervals $[\ulb_n,\uub_n]$,
$[\olb_n,\oub_n]$ and $[\nlb_n,\nub_n]$ for a range of cases.  It
highlights how modest are the gains from repeated observation of
homogeneous mediators and how alternating evidence can tighten bounds
as long as $\rho\neq 0$.

\begin{figure}[p!]
  \centering
  \includegraphics[width=0.9\linewidth]{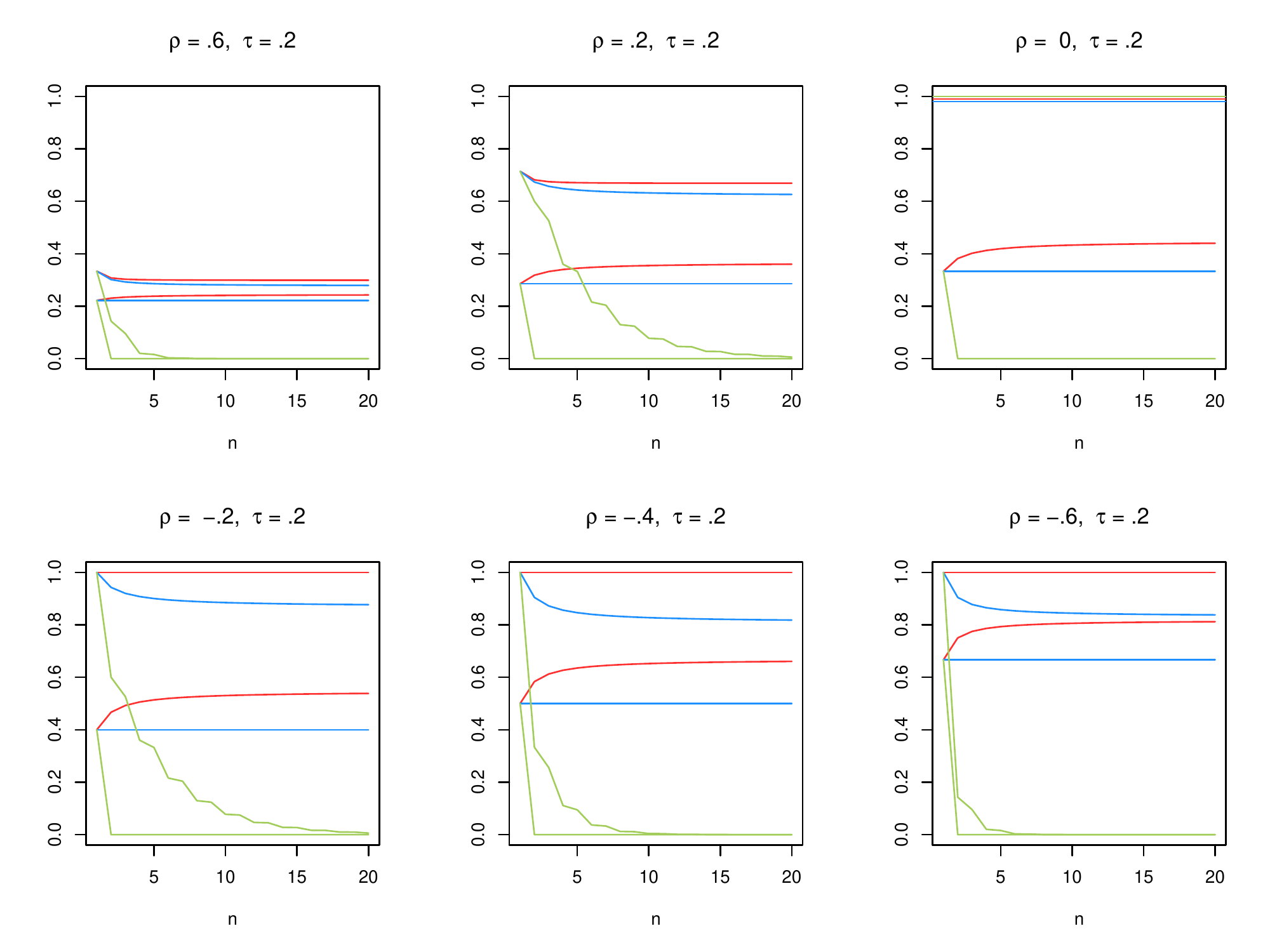}
  \caption{Bounds from homogeneous decompositions of length $n$ of
    $P=P(\tau,\rho)$, for $\tau = 0.2$ and six different values of
    $\rho$.  $\pc$ bounds when mediators not observed (marked blue),
    $\tpc$ bounds for mediators observed at 1 for homogeneous
    decompositions of length $n$ (red), and $\hpc$ bounds for
    alternating observations (green).  For positive data the bounds
    tighten only modestly as the number of links in the chain
    increases.  For alternating data the bounds converge to 0 unless
    $\rho = 0$.}
  \label{fig:withn}
\end{figure}

\subsubsection{Unboundedly many mediators}
\label{sec:unbound}

We now consider the behaviour of the bounds when we have a potentially
unlimited sequence of variables directly mediating between $X$ and
$Y$---still assuming identical one-step transition matrices.  Our
results are given in \thmref{lim}.

\begin{theorem}
  \label{thm:lim}
  \begin{eqnarray}
    \label{eq:lbinfu}
    {\ulb}_{\infty} := \lim_{n\rightarrow\infty}\ulb_n &=& \frac{2\tau}{1+\tau+\rho}\\
    {\uub}_\infty := \lim_{n\rightarrow\infty}\uub_n  &=&\frac {\tau + \tau^{|\sigma|}}{ 1 + \tau + \rho } \\
    \label{eq:lbinf}
    {\olb}_{\infty} := \lim_{n\rightarrow\infty}\olb_n  &=& \tau^{\half(1+\sigma)}\\
    \label{eq:obinf}
    {\oub}_\infty  := \lim_{n\rightarrow\infty}\oub_n   &=&\min\left\{1,\tau^{\sigma}\right\}\\
    {\nlb}_{\infty} := \lim_{n\rightarrow\infty}\nlb_n &=& 0\\
    {\nub}_{\infty} := \lim_{n\rightarrow\infty}\nub_n  &=& \left\{
                                                            \begin{array}[c]{cc}
                                                              0 & \text{if } \rho \neq 0\\ 
                                                              1 & \text{if } \rho = 0
                                                            \end{array} \right.   
  \end{eqnarray}
\end{theorem}

\begin{proof}
  See \appref{homproofs}.
\end{proof}


In particular, for $\rho=0$ we have
$$0= \nlb_{\infty} < \ulb_{\infty} = 2\tau/(1+\tau)\leq \olb_\infty = {\tau}^\half,$$
and
$$\uub_\infty = \oub_\infty = \nub_\infty =
1.$$

\begin{prop}
  \label{prop:inc}
  For $|\rho|< 1-\tau$, $\olb_n$ is a concave strictly increasing
  function of $n$, and $\uub_n$ and (for $\rho>0$) $\oub_n$ are both
  convex strictly decreasing functions of $n$.
\end{prop}
We do not have a full proof of \propref{inc}.  Supporting evidence is
given by numerous plots of $\olb_n$ and $\oub_n$ against $n$ for
various $(\tau,\rho)$ pairs, and the following two results, which are
proved in \appref{homproofs}.

\begin{lemma}
  \label{lem:big}
  If $|\rho|< 1-\tau$, then $\olb_n$ is a concave increasing function
  of $n$, and $\uub_n$ and (for $\rho>0$) $\oub_n$ are convex strictly
  decreasing functions of $n$, for $n$ sufficiently large.
\end{lemma}

\begin{lemma}
  \label{lem:double} For the non-degenerate case $|\rho|< 1-\tau$,
  $\uub_{2n} < \uub_n$, $\olb_{2n} > \olb_n$, and (for $\rho>0$)
  $\oub_{2n} < \uub_n$.
\end{lemma}

\subsection{Implications for data gathering}
\label{sec:strat}

Our results have focused on improving the bounds on $\pc$ by learning
about general mediating processes together with values for
prespecified mediators for the case at hand.  Our results can also be
used to suggest which mediators researchers might most fruitfully seek
to observe for the case at hand.

Thus consider a homogeneous process with $n$ steps ($n$ even) and
suppose that researchers can observe the value of just one mediator
$M_i$.  In this case we can show that the lower bound $\widetilde\lb$
on $\tpc$, if we were to observe $M_i=1$, is maximized if the
{central} mediator in the sequence is observed.  To see this, note
that from \eqref{taun}, \eqref{rhon} and \thmref{prodpc}, the lower
bound $\widetilde\lb$ from observation of mediator $M_k=1$ is given by
the product of the lower bound for the probability that $X=1$ caused
$M_k=1$ and the lower bound for the probability that $M_k=1$ caused
$Y=1$:
$$\frac{{2}(\tau')^k}{1+(\tau')^k+\rho'\{1+\tau'  + \dots +(\tau')^k\}}\times
\frac{{2}(\tau')^{n-k}}{1+(\tau')^{n-k}+\rho'\{1+\tau' + \dots
  +(\tau')^{n-k}\}}$$
where $\tau'$ and $\rho'$ are given by \eqref{tau} and \eqref{rho}.
This expression has the form ${c}/{f(k)f(n-k)}$, where $f(k)$ is
decreasing and convex in $k$: this holds since
$\Delta_{k+1} := f(k+1) - f(k) =
\tau'^{k+1}-\tau'^k+\rho'\tau'^{k+1}=\tau'^k(\tau'+\rho'-1)<0$, and
$\Delta_{k+1} - \Delta_{k} = (\tau'^k-\tau'^{k-1})(\tau'+\rho'-1)>0$.
Hence the denominator is minimised, and so $\widetilde\lb$ is
maximised, when $k=n-k$.

As an illustration, suppose 121 dominoes stand in a row.  The fall of
any domino increases the chance that its neighbor will fall from 0.005
to 0.995.  You know that the first domino was knocked and fell, that
the last is also down, and want the probability that the fall of the
first one caused the fall of the last one.  A lower bound above 50\%
would secure a conviction of domino~1.

With no further information, the lower bound is 0.461---not enough to
convict.  But now suppose you can seek information on the status of
just one other domino in the sequence: which should you choose?  It is
better to choose in the middle than at the edges.

If for example you were to seek information on the status of domino 2
and found that it had fallen, you would find
$\widetilde\lb = 0.463$---a modest gain, reflecting the fact that you
fully expected domino 2 to have fallen, given that domino 1 was
knocked.  However, you are less sure you will find domino 61 down.  If
you do, you find $\widetilde\lb = 0.501$--- enough to convict
domino~1.

Note that in all cases the lower bound would be 0 if the intermediate
domino were found to be standing.  Taking both possible outcomes into
account, the {expected} lower bound is always $0.461$.  But the second
strategy does better than the first, in allowing the possibility to
obtain a larger lower bound (albeit with a smaller probability), and
so secure a conviction.


\section{Comparisons with other bounds}
\label{sec:comp}

Although knowledge of mediators can narrow bounds, we have seen that
this narrowing can be modest, even with access to an infinite sequence
of positive evidence along a causal path.  To put our results in
context, we compare them with bounds that can be achieved from
monotonicity, and from covariate information.  Knowledge of the bounds
achievable by different strategies provides some guidance as to
whether a strategy would be worth pursuing.

\paragraph{Monotonicity}
Suppose that we somehow knew that there are no cases for which the
exposure would prevent the outcome, \ie, such that $Y_0 = 1, Y_1=0$.
From \tabref{y0y1} this is equivalent to $\xi=\tau$, its lower limit,
which in turn implies that $\pc$, given by \eqref{now}, is identified
at its lower limit, $\slb = (2\tau)/(1+\tau+\rho).$

However, since monotonicity is an attribute of the typically
unidentifiable joint distribution of $(Y_0,Y_1)$, it is not easy to
justify without additional knowledge.  One case where this works is
when we know the existence of a mediation process with decomposition
\eqref{ba1}.
  

\paragraph{Observed covariate}

Suppose that, in addition to $X$ and $Y$, we can observe a binary
covariate $C$, which can affect the dependence of $Y$ on $X$.  Let
$\pi= \Pr(C=1)$, and let $P_i$ be the transition matrix from $X$ to
$Y$, conditional on $C=i$; for consistency with the known
$P = P(\tau,\rho)$ we must have $P = \pi P_1 + (1-\pi)P_0$.

In particular, it could be the case that $\pi = (1+\tau-\rho)/2$, and
$$P_1  = \left( 
  \begin{array}{cc} 
    1 & 0 \\ 
    \frac{1-\tau-\rho}{2} & \frac{1+\tau+\rho}{2}
  \end{array}
\right) \qquad P_0 = \left(
  \begin{array}{cc}
    0 & 1\\
    \frac{1-\tau-\rho}{2} & \frac{1+\tau+\rho}{2}
  \end{array}\right).$$ 
In this case knowledge that an individual with $X=Y=1$ also has $C=1$
is enough to identify \pc at 1.

\paragraph{Unobserved covariate}


As shown in \citet{dawid2011role}, knowledge of covariates can improve
bounds, even if their values are not observed for the case at hand.
In particular, this can let us identify \pc at the upper bound,
$\sub=\min\{1, \frac{1+\tau - \rho}{1+\tau +\rho}\}$.  For this to be
possible, however, the average treatment effect must be negative for
some value of $C$.

Thus suppose $\pi = \frac{1+\tau+\rho}{2}$, and the conditional
transition matrices are:

For $\rho<0$,
$$P_{1}  = \left( 
  \begin{array}{cc} 
    1 & 0 \\
    0 & 1
  \end{array}
\right) \qquad P_0 = \left(
  \begin{array}{cc}
    \frac{-2\rho}{1-\tau-\rho} & \frac{1+\tau-\rho}{1-\tau-\rho} \\
    1 & 0 \end{array}
\right).$$
  
For $\rho\geq0$,
$$P_{1}  = 
\left(
  \begin{array}{cc}
    \frac{1 +\tau-\rho}{1 +\tau+\rho} & \frac{2\rho}{1+\tau+\rho} \\
    0 & 1
  \end{array}
\right) \qquad P_0 = \left(
  \begin{array}{cc} 0 & 1 \\
    1 & 0
  \end{array}
\right).$$

In either case, knowledge that $X=Y=1$ is sufficient to infer that
$C=1$.  This identifies the probability of causation: $\pc=1$ for
$\rho<0$, $\pc = \frac{1 +\tau-\rho}{1+\tau+\rho}$ for $\rho\geq 0$.
In both cases we hit the upper bound.

\paragraph{Comparisons}
Figure~\ref{fig:bounds} compares the bounds obtained, for a range of
values of $\tau$ and $\rho$.  It illustrates how, in general, lower
bounds rise with $\tau$ and fall with $\rho$.  For homogeneous
processes the lower bounds improve on the simple bounds, although the
gain from unlimited steps is not a striking improvement on that for
just two steps.  The best gains from non-homogeneous decompositions
are substantial, as are the gains from knowledge of covariates,
especially when $\rho$ is small.

\begin{figure}[htbp]
  \centering
  \includegraphics[width=0.8\linewidth]{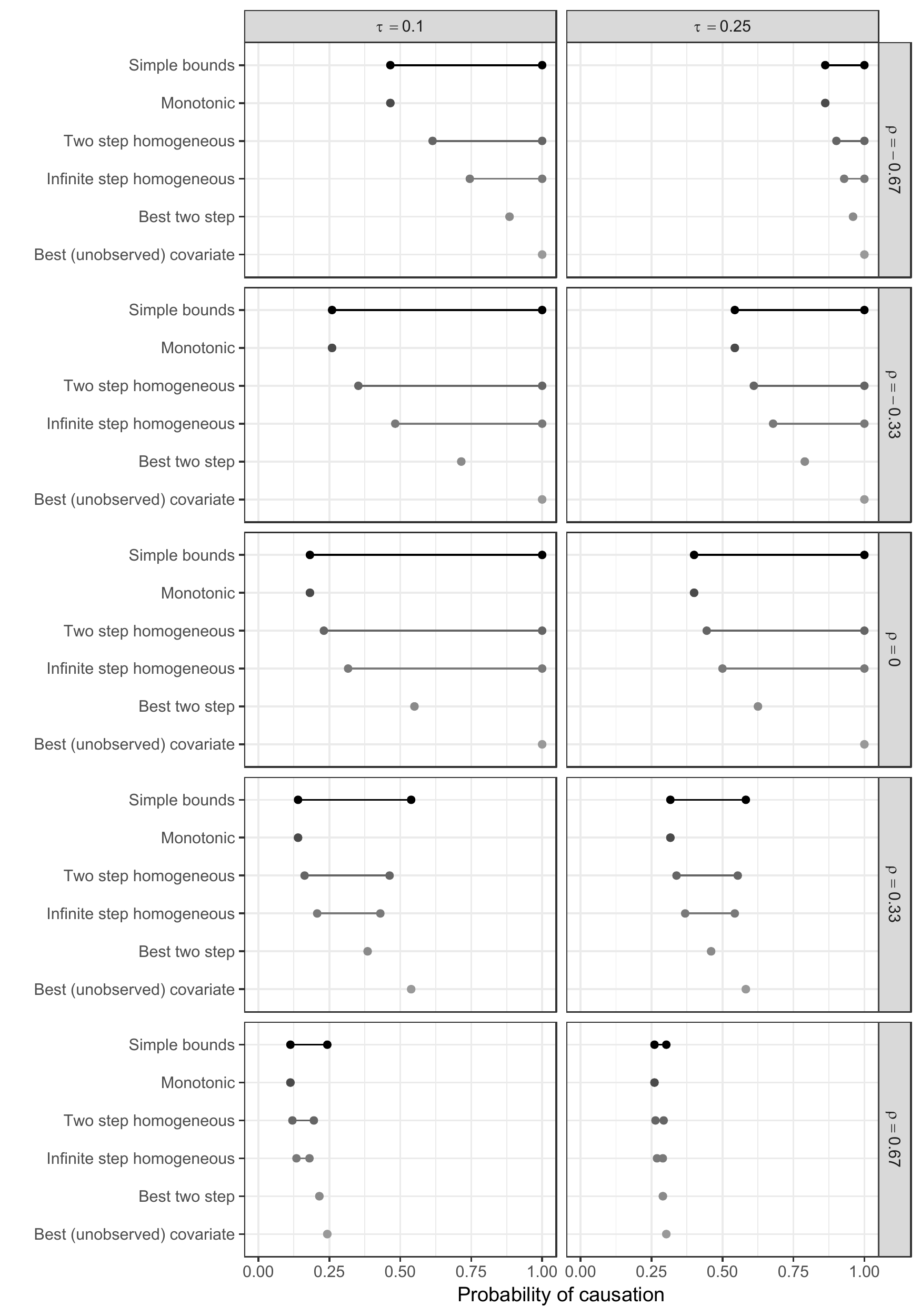}
  \caption{Comparison of bounds on \pc given different auxiliary
    information.  Simple bounds are given by \eqref{simple}.
    Monotonic bounds assume no prevention.  Homogeneous two-step and
    infinite-step bounds, assuming positive evidence observed at every
    step, are given by \eqref{pcbounds2}, \eqref{pcbounds3},
    \eqref{lbinf}, and \eqref{obinf}.  Best two-step bounds are as
    given by $\underline\uub$ or $\underline\oub$ of
    \tabref{resultstable}.  Best bounds based on an unobserved binary
    covariate are as described in \secref{comp}.}
  \label{fig:bounds}
\end{figure}

\section{Conclusion}
\label{sec:conc}

We close with some comments, which may help to guide the collection of
ancillary evidence to improve the bounds on the probability of
causation.  These are based on our general results, as exemplified in
Figure~\ref{fig:bounds}.

\renewcommand{\theenumi}{\arabic{enumi}}

\begin{enumerate}
\item Knowledge of mediation processes, and of positive values for
  some mediators in a particular case, can raise the lower bound on
  the probability of causation, thus providing some evidence against a
  sceptic who doubts that the outcome in the case can be attributed to
  the putative cause.  However, it may well not raise the bound enough
  to convince her.  In contrast, for some processes, observing
  negative evidence on mediators can effectively convince the sceptic
  that the outcome is {\em not\/} the result of the exposure.
\item Observing positive data on homogeneous mediation processes can
  improve the bounds, but there are diminishing returns, and full
  identification is not achieved, even with infinite data.
\item For a homogeneous process, observation in the middle of the
  process is more informative than nearer the edges.
\item Heterogeneous mediation processes can sometimes yield
  identification with minimal auxiliary data gathering:
  \begin{itemize}
  \item A process where $X$ is a necessary condition for a sufficient
    condition for $Y$ yields the largest possible upper bound, and
    identifies the probability of causation.  For example, if it is
    known that the effect of delivering a deworming medicine passes
    uniquely through ingestion, and ingestion is sufficient for
    effective deworming, then evidence of ingestion raises the lower
    bound and identifies the probability of causation.
  \item A process in which $X$ is a sufficient condition for a
    necessary condition for $Y$ yields identification, and {there is
      no gain from gathering data on the mediator}.  For instance if
    ingesting medicine is a sufficient condition for good health, and
    good health is a necessary condition for good school performance,
    then observing ingestion and good school performance is sufficient
    to achieve identification.  There are no additional gains from
    measuring health, since good health is already implied by good
    performance.
  \end{itemize}
\item Potential gains from knowledge of mediation processes are
  typically weaker than potential gains from knowledge of conditions
  under which interventions are more or less effective.  Even when
  covariates are unobserved for the case at hand, knowledge of the
  general effect of covariates can tighten the bounds when some
  subgroups exhibit adverse effects.  On this basis researchers might
  be able to assess whether a search for a suitable covariate could
  lead to improved bounds, and perhaps even identification of the
  probability of
  causation.
\end{enumerate}

\newpage

\renewcommand{\thesection}{A.\arabic{section}} \setcounter{section}{0}
\renewcommand{\theequation}{A.\arabic{equation}}
\setcounter{equation}{0}

\begin{appendices}


	
	
	

  \section{Proof of \thmref{main}}
  \label{sec:mainproof}
  \subsection{Mediators unobserved}

  \begin{description}
  \item[Lower bounds:] $\ulb$ is unchanged by knowledge of the
    mediation process alone and so the largest and smallest values of
    $\ulb$ are $\underline{\ulb} = \overline{\ulb} = \slb$.

  \item[Smallest upper bound:] From \eqref{simpleu} we can see that
    for a degenerate two-term decomposition with $|\rho_1| = 1-\tau_1$
    and $|\rho_2| = 1-\tau_2$, $\uub=\ulb=\slb$.  In this case $\pc$
    is identified.

  \item[Largest upper bound:] It follows from \corref{ubusmall} and
    \eqref{now} that this is achieved when there are no mediators, and
    is thus $\sub$.

  \end{description}

%
%
%
%
%
%
%

  \subsection{Positive data observed at every step}

  We now consider the case where mediators are observed.  Then, for
  the decomposition \eqref{arb},
  \begin{eqnarray*}
    P &=& P_1 \times P_2 \times \ldots \times P_n\\
    \tpc &=& \prod_{i=1}^n \pc_i\\
    \olb &=& \prod_{i=1}^n \slb_i\\
    \oub &=& \prod_{i=1}^n \sub_i
  \end{eqnarray*}

  \begin{description}
  \item[Smallest lower bound]
    \quad\\
    It follows from \lemref{lbineq} that the smallest achievable lower
    bound is
$$\slb = 2\tau/(1+\tau+\rho),$$
which does not require any mediators.

\item[Smallest upper bound]
  \quad\\
  Trivially we must have $\oub \geq \olb \geq \slb$.

  Note now that the decomposition \eqref{ba1} identifies
  $\tpc = \slb$, whence in particular $\oub = \slb$, the smallest
  possible value.
  
\item[Largest lower bound]
  \quad\\
  \begin{lemma}
    \label{lem:lineq}
	$$\frac{2\tau}{1+\tau+\rho} \leq \frac{1+\tau-\rho}2.$$
      \end{lemma}
      \begin{proof}
	This holds since
	$$(1+\tau+\rho)(1+\tau-\rho) -4\tau = (1-\tau)^2 -\rho^2 \geq 0.$$
      \end{proof}

      \begin{lemma}
	\label{lem:r1}
	Let $P = P_1 \times P_2$.  Then
	$$\frac{1+\tau_1-\rho_1}2 \times \frac{1+\tau_2-\rho_2}2 \leq \frac{1+\tau-\rho}2.$$ 
      \end{lemma}
      \begin{proof}
        Follows from matrix multiplication, on noting that each term
        is the leading entry of its associated transition matrix.
      \end{proof}
      \begin{cor}
        \label{cor:prod}
        Let $P = P_1 \times \ldots \times P_n$.  Then
	$$\prod_{i=1}^n \frac{1+\tau_i-\rho_i}2 \leq \frac{1+\tau-\rho}2.$$ 
      \end{cor}
     
      From \eqref{prodbounds}, \lemref{lineq} and \corref{prod} we
      deduce:
      \begin{cor}
	\label{cor:lbless}
	Let $P = P_1 \mid \ldots \mid P_n$.  Then
        $\olb \leq (1+\tau-\rho)/2$.
      \end{cor}
      However the value $\olb = (1+\tau-\rho)/2$ can be achieved,
      specifically for the degenerate two-term decomposition
      \eqref{ab1}, so this is indeed the largest lower bound.  And in
      this case we have identification: $\tpc = (1+\tau-\rho)/2$.

      We note that, since $\rho\leq 1-\tau$, the largest lower bound,
      $(1+\tau-\rho)/2$, can not exceed the simple upper bound
      $\sub=(1+\tau-\rho)/(1+\tau+\rho)$.  Thus any lower bound must
      lie in the simple interval $[\slb,\sub]$.

    \item[Largest upper bound]
      \quad\\
      \begin{lemma}
	\label{lem:rhoineq}
	$$\min\left\{1,\frac{1+\tau-\rho}{1+\tau+\rho}\right\} \leq \min\{1,1-\rho\}.$$
      \end{lemma}
      \begin{proof}
	Trivial if $\rho\leq 0$.  Otherwise follows from
        $(1-\rho)(1+\tau+\rho) - (1+\tau-\rho)= \rho(1-\tau-\rho) \geq
        0$.
      \end{proof}
      \begin{lemma}
	\label{lem:rr}
	Let $P = P_1 \times P_2$.  Then
	$$\min\{1,1-\rho_1\} \times \min\{1,1-\rho_2\} \leq\min\{1,1-\rho\}.$$
      \end{lemma}
      \begin{proof}
        Trivial if both $\rho_1$ and $\rho_2$ (and hence, by
        \eqref{rho2} and the fact that $\tau_2>0$, also $\rho$) are
        negative.
	
        If $\rho_1 \leq 0$, $\rho_2\geq 0$, we have to show
        $\rho_2 \geq \rho$.  This follows from \eqref{rho2}.
        Similarly if $\rho_1 \geq 0$, $\rho_2\leq 0$ (using
        $\tau_2 \leq 1)$.
	
        Finally, if $\rho_1 > 0$, $\rho_2> 0$ (and so also $\rho>0)$,
        the result follows from \eqref{mod2}.
      \end{proof}

      \begin{cor}
        \label{cor:mult}
        Let $P = P_1 \times \ldots \times P_n$.  Then
	$$\prod_{i=1}^n \min\{1,1-\rho_i\} \leq\min\{1,1-\rho\}.$$
      \end{cor}

      From \lemref{rhoineq} and \eqref{prodbounds} we deduce:
      \begin{cor}
	\label{cor:ubless}
	For decomposition $P = P_1 \mid \ldots \mid P_n$,
        $\oub \leq \min\{1,1-\rho\}$.
      \end{cor}

      However the value $\ub = \min\{1,1-\rho\}$ can be achieved.  If
      $\rho\leq 0$, no mediators are required.  If $\rho > 0$, the
      value $\oub = 1-\rho$ is achieved by the two-term decomposition
      \eqref{C}.  By \lemref{rhoineq}, this largest upper bound
      $1-\rho$ is at least as large as the simple upper bound $\sub$
      of \eqref{simple}.

      Since we know that $\olb \leq \sub$, we cannot have
      identification of $\tpc$ in this case unless these inequalities
      become equalities, which only holds when $\rho=1-\tau$.  In fact
      for the decomposition \eqref{C} we have $\olb = \tau$.
    \end{description}

\subsection{Negative data observed at some steps}
The lower bounds at 0 are immediate from Equation \eqref{prodtilde}.
It is easy to verify that the lowest upper bound at 0 is achievable by
the decomposition \eqref{C}, and the highest upper bound at 1 is
achievable from the decompositions in \eqref{ext1} and \eqref{ext2}.
Since these bounds are at 0 and 1 they are the extreme values
obtainable from any process involving some negative data.

\section{Proof of \thmref{mixed}}
\label{sec:mixproofs}

Using \eqref{s00++}--\eqref{s11++} and \eqref{s00+-}--\eqref{s11+-},
we have the following upper bounds for a single step.

For $\rho \geq 0$:
\begin{eqnarray*}
  \ub(00) &=& 1\\
  \ub(01) &=& \gamma'=\gamma\\
  \ub(10) &=& 1\\
  \ub(11) &=& \delta'
\end{eqnarray*}

For $\rho < 0$:
\begin{eqnarray*}
  \ub(00) &=& \delta'\\
  \ub(01) &=& 1\\
  \ub(10) &=& \gamma\\
  \ub(11) &=& 1
\end{eqnarray*}

When $\rho=0$, all these upper bounds are 1, and the upper bound for
any evidence sequence is 1.

Otherwise, $\gamma<1$ does not depend on $n$, while
$\delta'\rightarrow 1$ as $n\rightarrow\infty$.  So there exists
$N\geq 2$ such that $n>N \Rightarrow \gamma' < (\delta')^2$.
Henceforth we suppose $n>N$.  Then we have:
\begin{lemma}
  \label{lem:mixup}
  \begin{eqnarray*}
    \ub(111) &>& \ub(101)\\
    \ub(000) &>& \ub(010)
  \end{eqnarray*}
\end{lemma}

\begin{proof}
  We have $\ub(111) = \{\ub(11)\}^2 = (\delta')^2$ if $\rho\geq 0$, or
  1 if $\rho< 0$, while $\ub(000) = \{\ub(00)\}^2=1$ if $\rho \geq 0$,
  or $(\delta')^2$ if $\rho< 0$.  In all cases
  $\ub(101) = \ub(010) = \ub(10)\ub(01) = \gamma$, while $\ub(111)$
  and $\ub(000) \geq (\delta')^2 > \gamma$.
\end{proof}

\begin{cor}
  \label{cor: no3}
  The optimal sequence $\bs_0$ can not contain any subsequence of
  repeated values of length greater than 2.
\end{cor}

We now consider separately the cases of positive and negative $\rho$.

\begin{description}
\item[1.  $\rho > 0$.]
  \quad\\
  Suppose $\bs_0$ contains a subsequence $00$.  It must then be
  followed by a $1$, so $\bs_0$ contains a subsequence $001$.  Since
  $\ub(001) = \gamma$, while $\ub(011)= \gamma\delta' < \gamma$, on
  replacing this subsequence by $011$ we would achieve a smaller upper
  bound.  This contradiction shows that $\bs_0$ cannot contain any
  successive repeated $0$'s.

  Now suppose $\bs_0$ contains a subsequence $11$.  Consider the first
  appearance of this.  If not at the very end, it must be followed by
  $01$.  Now replace this subsequence $1101$ by $1011$.  Since
  $\ub(1101) = \ub(1011) = \gamma\delta'$, we have not changed
  $\ub(\bs_0)$, but have postponed the first occurrence of $11$.  We
  can thus assume that the first such occurrence (if any) is at the
  very end.

  If now $n$ is even, the first $n-1$ values must be the alternating
  sequence $1010\ldots 1$.  But this can not be followed by $11$,
  since that would produce a subsequence $111$.  We deduce that
  $\bs_0$ must be the full alternating sequence $1010\ldots 1$.  The
  smallest possible upper bound with mixed evidence is thus
  $\nub_n = \ub(\bs_0) = \gamma^{n/2}$.

  If $n$ is odd, there must at least one appearance of $11$.  The
  above argument now delivers as $\bs_0$ the alternating sequence
  $1010\ldots 1$ of length $n$, followed by the final $1$.  We now
  have $\nub_n = \ub(\bs_0) = \gamma^{(n-1)/2}\delta'$.

\item[2.  $\rho <0$.]
  \quad\\
  The argument here is almost, but not quite, a mirror image of that
  above.

  Suppose that $\bs_0$ contains a subsequence $11$.  If not at the
  very end, it must be followed by a $0$ so that $s_0$ contains a
  subsequence $110$.  Now, since $\ub(110)= \gamma$ and
  $\ub(100)=\gamma \delta' < \gamma$, $\bs_0$ cannot contain any
  internal successive repeated 1's.  Also, $\bs_0$ can not end with
  $11$, and hence with $011$, since
  $\ub(001) =\delta’ < \ub(011) = 1$.  So there can be no repeated
  $1$'s.

  Now suppose $\bs_0$ contains a subsequence $00$.  Consider the first
  appearance of this.  If not just before the final $1$, it must be
  followed by $10$.  Now replace this subsequence $0010$ by $0100$.
  Since $\ub(0010) = \ub(0100) = \gamma\delta'$, we have not changed
  $\ub(\bs_0)$, but have postponed the first occurrence of $00$.  So
  the only possibility for two successive $0$s is if $\bs_0$ ends with
  $001$.

  Before the end, we must have an alternating sequence
  $101 \ldots 01$.

  If $n$ is even, we can not then conclude with $001$, so in this case
  we must have the full alternating sequence $1010\ldots 1$.  The
  smallest possible upper bound with mixed evidence is, again,
  $\nub_n = \ub(\bs_0) = \gamma^{n/2}$.

  If $n$ is odd we must have the alternating sequence $1010\ldots 1$
  of length $n-2$, followed by $001$.  We again find
  $\nub_n = \ub(\bs_0) = \gamma^{(n-1)/2}\delta'$.
\end{description}
  
\section{Proofs for \secref{unbound}}
\label{sec:homproofs}

\begin{proof0}{\thmref{lim}}
  Using {\tt Mathematica} \citep{Mathematica}, we obtain expansions
  \begin{eqnarray}
    \label{eq:explb}
    \olb_n &=& \olb_\infty \times \{1 + k/n + O(1/n^2)\}\\
    \label{eq:expub}
    \oub_n &=& \oub_\infty \times \{1 + q/n^2 + O(1/n^3)\}\qquad(\mbox{for }\rho>0)
  \end{eqnarray}
  with
  \begin{eqnarray*}
    \olb_\infty &=& \tau^{\frac{1-\tau+\rho}{2(1-\tau)}}\\
    \oub_\infty &=& \tau^{\frac{\rho}{1-\tau}}.
  \end{eqnarray*}
	
  The expression for $\uub_\infty$ is obtained similarly.
	
  Finally, since $\ulb_n = \slb$ for all $n$, we trivially have
  $\ulb_\infty = \slb$.
\end{proof0}

\begin{proof0}{\lemref{big}}
  In \eqref{explb} and \eqref{expub}, {\tt Mathematica} gives
  \begin{eqnarray*}
    k &=& -\frac{\left\{(1-\tau)^2-\rho^2\right\}\ln ^2(\tau)}{8(1-\tau)^2}\\
    q &=& -\frac{\rho \left\{(1-\tau)^2- \rho^2)\right\} \ln ^3(\tau)}{12 (1-\tau)^3}.
  \end{eqnarray*}

  In particular $k < 0$, $q > 0$.  Thus
$$d_n := \olb_{n+1} - \olb_n = - \olb_\infty\, k/n^2 + O(1/n^3)$$
where the leading term is positive, so $\olb_n$ is eventually
increasing.  Similarly
$$d_{n+1}-d_n =  2\, \olb_\infty\, k/n^3 + O(1/n^4)$$
with negative leading term, so $\olb_n$ is eventually concave in $n$.
A similar argument shows that, for $\rho>0$, $\oub_n$ is eventually
decreasing and convex in $n$.  We note that the convergence of
$\oub_n$ to its limit is at a faster rate than for $\olb_n$.

The behaviour of $\uub_n$ is obtained similarly (the limit being
approached at rate $1/n$).
\end{proof0}



%
%
%
%




\begin{proof0}{\lemref{double}}	  	
  Consider the $n$-part homogeneous decomposition
  $P = P_1\mid \ldots \mid P_n$.  Now replace each $P_i$ by its
  homogeneous 2-part decomposition $P_i = Q_{i1} \mid Q_{i2}$, so
  creating the $2n$-part homogeneous decomposition
  $P = Q_{11}\mid Q_{12}\mid \ldots\mid Q_{n1}\mid Q_{n2}$.
	
  By \corref{ubusmall} and \eqref{now} we see that $\uub$ decreases on
  making these replacements.
	
  The argument of \secref{specobs} shows that $\olb$ is increased by
  these replacements.
	
  To show the result for $\oub$ it is enough to show that
  $\oub < \sub$ for a two-term homogeneous decomposition with
  $\rho>0$.  That is to say,
  \begin{displaymath}
    \frac{(1+\tau'-\rho')^2}{(1+\tau'+\rho')^2} < \frac{1+\tau-\rho}{1+\tau+\rho}
  \end{displaymath}
  or equivalently
  \begin{displaymath}
    2\rho\left\{(1+\tau')^2+ \rho'^2\right\} <  4\rho'(1 +\tau')(1+\tau).
  \end{displaymath}  
  Noting that $\rho = \rho'(1+\tau')$ and $\tau = \tau'^2$, this
  becomes
  \begin{displaymath}
    \left(1+\tau'\right)^2+ \rho'^2  <  2\left(1+\tau'^2\right),
  \end{displaymath}
  equivalent to $\rho'^2 < (1- \tau')^2$, which holds since, by
  \eqref{tau} and \eqref{rho},
  $\rho'/(1-\tau') = \rho/(1-\tau) \in (0,1)$ by assumption.
	
\end{proof0}

\end{appendices}

\newpage

\bibliographystyle{apsr} \bibliography{strings,bib}

\end{document}